\documentclass[
english
]{amsart}

\usepackage{amssymb}
\usepackage{amsthm}
\usepackage{amsfonts}
\usepackage{babel}
\usepackage{amsmath}
\usepackage{longtable} 
\usepackage{array} 
\usepackage{setspace}
\usepackage{datetime}
\usepackage{ifpdf}
\usepackage{textcomp}
\usepackage{eufrak}  
\usepackage{euscript}
\usepackage{url}
\usepackage{enumerate}

\usepackage{pslatex}          
\usepackage[pdftex]{graphicx}
\usepackage[pdftex,breaklinks]{hyperref} 

\newtheorem{theorem}{Theorem}[section]
\newtheorem{lemma}[theorem]{Lemma}

\theoremstyle{definition}
\newtheorem{definition}[theorem]{Definition}

\numberwithin{equation}{section}

\def \eN{\mbox{E.~Ntienjem}}
\def \waS{\mbox{W.~A.~Stein}}
\def \mN{\mbox{M.~Newman}}

\def \gK{\mbox{G.~K\"{o}hler}}
\def \ljpK{\mbox{L.~J.~P.~Kilford}}
\def \aA{\mbox{A.~Alaca}}
\def \sA{\mbox{\c{S}.~Alaca}}

\def \ksW{\mbox{K.~S.~Williams}}

\def \eR{\mbox{E.~Royer}}

\def \nK{\mbox{N.~Koblitz}}

\def \E{\mbox{$ I\!\! E$}}
\def \M{\mbox{$ I\!\! M$}}

\hypersetup{
pdfinfo={
   Author={Eb\'{e}n\'{e}zer Ntienjem},
   Title={Evaluation of the Convolution Sum involving the Sum of 
Divisors Function for 44 and 52},
   CreationDate={D:20160527003000},
   ModDate={D:\pdfdate)=},
   Subject={2010 Mathematics Subject Classification: 11A25, 11E20, 11E25, 11F11, 11F20, 11F27},
   Keywords={Sums of Divisors function; Dedekind eta function; Convolution Sums; 
Modular Forms; Eisenstein forms; Cusp Forms; Octonary quadratic Forms; 
Number of Representations}
}
}

\author{Eb\'{e}n\'{e}zer Ntienjem}
 \address{
Centre for Research in Algebra and Number Theory \\
 School of Mathematics and Statistics\\
 Carleton University\\
 1125 Colonel By Drive\\
 Ottawa, Ontario, K1S 5B6, Canada
}
\email{ebenezer.ntienjem@carleton.ca;ntienjem@gmail.com}


\begin{document}

\title[Evaluation of the Convolution Sums for 44 and 52]
{Evaluation of the Convolution Sums 
$\underset{\substack{
 {(l,m)\in\mathbb{N}_{0}^{2}} \\ {\alpha\,l+\beta\,m=n}
} }{\sum}\sigma(l)\sigma(m)$, 
where $\alpha\beta=44,52$
}

\begin{abstract}
The convolution sum, 
$\underset{\substack{
 {(l,m)\in\mathbb{N}_{0}^{2}} \\ {\alpha\,l+\beta\,m=n}
} }{\sum}\sigma(l)\sigma(m)$, 
where $\alpha\beta=44,52$, is evaluated for all natural numbers $n$. 
We then use these convolution sums to determine formulae for the number 
of representations of a natural number 
by the octonary quadratic forms
$a\,(x_{1}^{2} + x_{2}^{2} + x_{3}^{2} + x_{4}^{2})+ b\,(x_{5}^{2} + x_{6}^{2}
  + x_{7}^{2} + x_{8}^{2})$, 
where $(a,b)= (1,11),(1,13)$. 
\end{abstract}

\maketitle

\section{Introduction} \label{introduction}

The sets of
natural numbers, non-negative integers, integers, rational numbers, real 
numbers and complex numbers, are denoted by $\mathbb{N}$, $\mathbb{N}_{0}$, 
$\mathbb{Z}$, $\mathbb{Q}$, $\mathbb{R}$ and $\mathbb{C}$, respectively. 

Suppose that $k,n\in\mathbb{N}$. We define the sum of positive divisors 
of $n$ to the power of $k$, $\sigma_{k}(n)$, by 
\begin{equation} \label{def-sigma_k-n}
 \sigma_{k}(n) = \sum_{0<d|n}d^{k}.
 \end{equation} 
We  write $\sigma(n)$ as a synonym for $\sigma_{1}(n)$ and we set 
$\sigma_{k}(m)=0$ if $m\notin\mathbb{N}$. 

The convolution sum, 
 $W_{(\alpha,\beta)}(n)$, is defined for all 
$\alpha,\beta\in\mathbb{N}$ such that $\alpha\leq\beta$ 
 as follows: 
 \begin{equation} \label{def-convolution_sum}
 W_{(\alpha, \beta)}(n) =  \sum_{\substack{
 {(l,m)\in\mathbb{N}_{0}^{2}} \\ {\alpha\,l+\beta\,m=n}
} }\sigma(l)\sigma(m).
 \end{equation}
We write $W_{\beta}(n)$ as a short hand for $W_{(1,\beta)}(n)$.

For those convolution sums 
$W_{(\alpha, \beta)}(n)$ that have so far been evaluated, the values of 
$(\alpha,\beta)$ are given in \autoref{introduction-table-1}. 
We evaluate the convolution sums for $(\alpha,\beta)=(1,44)$, $(4,11)$, 
$(1,52)$, $(4,13)$,
i.e., $\alpha\beta=44$ and $\alpha\beta=52$. The evaluation of these 
convolution sums have not been done yet according to 
\autoref{introduction-table-1}. 

Let $a,b,c,d\in \mathbb{N}$ be such that $\gcd(a,b)=1$ and $\gcd(c,d)=1$. 
The convolution sums are generally used to
determine explicit formulae for the number of representations of a positive 
integer $n$ by the octonary quadratic forms 
\begin{equation} \label{introduction-eq-1}
a\,(x_{1}^{2} +x_{2}^{2} + x_{3}^{2} + x_{4}^{2})+ b\,(x_{5}^{2} + x_{6}^{2} + 
x_{7}^{2} + x_{8}^{2}),
\end{equation}
and 
\begin{equation} \label{introduction-eq-2}
c\,(x_{1}^{2} + x_{1}x_{2} + x_{2}^{2} + x_{3}^{2} + x_{3}x_{4} + x_{4}^{2})
+ d\,(x_{5}^{2} + x_{5}x_{6}+ x_{6}^{2} + x_{7}^{2} + x_{7}x_{8}+ 
x_{8}^{2}),
\end{equation}
respectively. 

We use the evaluated convolution sums and other known convolution sums to 
determine formulae for the number of representations 
of a positive integer $n$ by the octonary quadratic form
\autoref{introduction-eq-1} for which $(a,b)=(1,11),
(1,13)$. These number of representations are also new according 
to \autoref{introduction-table-rep2} which displays known explicit 
formulae for the number of representations 
of $n$ by the octonary form \autoref{introduction-eq-1}.

We have organized this paper as follows. 
In \hyperref[modularForms]{Section \ref*{modularForms}} we briefly discuss 
modular forms and define eta functions and convolution sums. Then in 
\hyperref[convolution_44_52]{Section \ref*{convolution_44_52}} we 
discuss our main results on the evaluation of the convolution sums;  
the main results on the formulae for the number of representations of a positive 
integer $n$ are given in 
\hyperref[representations_44_52]{Section \ref*{representations_44_52}}. 

We use a software for symbolic scientific computation to obtain the results 
of this paper. The open source software packages 
\emph{GiNaC}, \emph{Maxima}, \emph{REDUCE}, \emph{SAGE} and the commercial 
software package \emph{MAPLE} build this software.


\section{Preliminaries} \label{modularForms}

We consider the upper half-plane,  
$\mathbb{H}=\{ z\in \mathbb{C}~ | ~\text{Im}(z)>0\}$, 
and the group  $G=\text{SL}_{2}(\mathbb{R})$ of $2\times 2$-matrices 
\begin{math}\left(\begin{smallmatrix} a & b \\ c &
 d \end{smallmatrix}\right)\end{math} such that $a,b,c,d\in\mathbb{R}$ and 
$ad-bc=1$. Let $\Gamma=\text{SL}_{2}(\mathbb{Z})$ be a 
subset of $G$ and let $N\in\mathbb{N}$. Then  
\begin{eqnarray*}
\Gamma(N) & = \bigl\{~\left(\begin{smallmatrix} a & b \\ c &
  d \end{smallmatrix}\right)\in\text{SL}_{2}(\mathbb{Z})~ |
  ~\left(\begin{smallmatrix} a & b \\ c &
 d\end{smallmatrix}\right)\equiv\left(\begin{smallmatrix} 1 & 0 \\ 0 & 1 
 \end{smallmatrix}\right) \pmod{N} ~\bigr\}
\end{eqnarray*}
is a subgroup of $\Gamma$. The subgroup $\Gamma(N)$ is called a \emph{principal congruence subgroup of 
level N}. If a subgroup $H$ of $G$ contains $\Gamma(N)$, then it is a 
\emph{congruence subgroup of level N}.

For our purpose we consider the congruence subgroup 
 \begin{align*}
\Gamma_{0}(N) & = \bigl\{~\left(\begin{smallmatrix} a & b \\ c &
  d \end{smallmatrix}\right)\in\text{SL}_{2}(\mathbb{Z})~ | ~
   c\equiv 0 \pmod{N} ~\bigr\}. 
\end{align*}
Let $k\in\mathbb{Z}, \gamma\in\Gamma$ and $f^{[\gamma]_{k}} : 
\mathbb{H}\cup\mathbb{Q}\cup\{\infty\} \rightarrow 
\mathbb{C}\cup\{\infty\}$ be   
the function whose value at $z$ is 
$f^{[\gamma]_{k}}(z)=(cz+d)^{-k}f(\gamma(z))$. 
The following definition is based on \nK 's textbook \cite[p. 108]{koblitz-1993}.
\begin{definition} \label{modularForms-defin-2}
Let $N\in\mathbb{N}$, $k\in\mathbb{Z}$, $f$ be a meromorphic function 
on $\mathbb{H}$ and $\Gamma'\subset\Gamma$ be a congruence subgroup of 
level $N$. 
\begin{enumerate}
\item[(a)] $f$ is a \emph{modular function of weight $k$} for 
$\Gamma'$ if
\begin{enumerate}
\item[(a1)] 
$f^{[\gamma]_{k}}=f$ for all $\gamma\in\Gamma'$,  
\item[(a2)] for any $\delta\in\Gamma$ it holds that $f^{[\delta]_{k}}(z)$ 
can be expressed in the form 
$\underset{n\in\mathbb{Z}}{\sum}a_{n}e^{\frac{2\pi i z n}{N}}$, 
wherein $a_{n}\neq 0$ for finitely many $n\in\mathbb{Z}$ such that $n<0$. 
\end{enumerate}
\item[(b)] $f$ is a \emph{modular form of weight $k$} for 
$\Gamma'$ if
	\begin{enumerate}
	\item[(b1)] $f$ is a modular function of weight $k$ for $\Gamma'$,
	\item[(b2)] $f$ is holomorphic on $\mathbb{H}$,
	\item[(b3)] $a_{n}=0$ for all $\delta\in\Gamma$ and for all $n\in\mathbb{Z}$ 
such that $n<0$.
	\end{enumerate}
\item[(c)] $f$ is a \emph{cusp form of weight $k$ for $\Gamma'$} if
	\begin{enumerate}
	\item[(c1)] $f$ is a modular form of weight $k$ for $\Gamma'$, 
	\item[(c2)] $a_{0}=0$ for all $\delta\in\Gamma$.
	\end{enumerate}
\end{enumerate}
\end{definition}
Let $k,N\in\mathbb{N}$. We denote by $\M_{k}(\Gamma_{0}(N))$ be the space of 
modular forms of weight $k$ for 
$\Gamma_{0}(N)$, $\EuScript{S}_{k}(\Gamma_{0}(N))$ the subspace of 
cusp forms of weight $k$ for $\Gamma_{0}(N)$, 
and $\E_{k}(\Gamma_{0}(N))$ the subspace of Eisenstein forms of 
weight $k$ for $\Gamma_{0}(N)$. 
In W.~A.~Stein's book (online version) \cite[p.~81]{wstein} it is shown that 
$\M_{k}(\Gamma_{0}(N)) =
\E_{k}(\Gamma_{0}(N))\oplus\EuScript{S}_{k}(\Gamma_{0}(N))$.  

According to Section 5.3 of W.~A.~Stein's book \cite[p.~86]{wstein} 
$E_{k}(q) = 1 - \frac{2k}{B_{k}}\,\underset{n=1}{\overset{\infty}{\sum}}\,
\sigma_{k-1}(n)\,q^{n}$, where $B_{k}$ are the Bernoulli numbers, if the 
primitive Dirichlet characters are trivial and $2\leq k$ is even.

We only consider trivial primitive Dirichlet characters and $4\leq k$ even in the 
sequel. Based on this consideration Theorems 5.8 and 5.9 in Section 5.3 of 
\cite[p.~86]{wstein} also hold.

\subsection{Eta Functions}  \label{etaFunctions}

On the upper half-plane $\mathbb{H}$ the Dedekind eta function, $\eta(z)$,  
is defined by 
$\eta(z) = e^{\frac{2\pi i z}{24}}\overset{\infty}{\underset{n=1}{\prod}}(1 - 
e^{2\pi i n z})$.
When we set $q=e^{2\pi i z}$, then 
\begin{equation*}
\eta(z) = q^{\frac{1}{24}}\overset{\infty}{\underset{n=1}{\prod}}(1 - q^{n}) 
= q^{\frac{1}{24}} F(q),\qquad
\text{ where } F(q)=\overset{\infty}{\underset{n=1}{\prod}}(1 - q^{n}).
\end{equation*}
 
\ljpK 's book 
\cite[p.~99]{kilford} and \gK 's book \cite[p.~37]{koehler} have a proof 
of the following theorem which  
we will apply to determine eta functions which 
belong to $\M_{k}(\Gamma_{0}(N))$, and particularly those eta functions 
that belong to $\EuScript{S}_{k}(\Gamma_{0}(N))$. As noted by \aA\ et al.\ 
\cite{alaca_alaca_ntienjem2016} credit to this theorem also goes to 
\mN\ \cite{newman_1957,newman_1959}.
\begin{theorem}[M.~Newman and G.~Ligozat] \label{ligozat_theorem} 
Let $N\in \mathbb{N}$  and let  
$f(z)=\overset{}{\underset{1\leq\delta|N}{\prod}}\eta^{r_{\delta}}(\delta z)$ 
be an eta function which satisfies the following conditions: 

\begin{tabular}{llll}
{\textbf{(i)}} & $\overset{}{\underset{1\leq\delta|N}{\sum}}\delta\,r_{\delta} 
	\equiv 0 \pmod{24}$, & 
{\textbf{(ii)}} &  $\overset{}{\underset{1\leq\delta|N}{\sum}}\frac{N}{\delta} 
	\,r_{\delta} \equiv 0 \pmod{24}$, \\
{\textbf{(iii)}} &  $\overset{}{\underset{1\leq\delta|N}
{\prod}}\delta^{r_{\delta}}$ \text{ is a square in } $\mathbb{Q}$, & 
{\textbf{(iv)}} &  $k=\frac{1}{2}\overset{}{\underset{1\leq\delta|N}
	{\sum}}r_{\delta}$ \text{ is an even integer,} \\
\end{tabular}

\textbf{(v)} for each positive divisor $d$ of $N$, the inequality 
	$\overset{}{\underset{1\leq\delta|N}{\sum}}\frac{\text{gcd}(\delta,d)^{2}}	
	{\delta} r_{\delta} \geq 0$ holds.

Then $f(z)\in \M_{k}(\Gamma_{0}(N))$. 

If \textbf{(v)} is replaced by 

\textbf{(v')} for each positive divisor $d$ of $N$ the inequality 
	$\overset{}{\underset{1\leq\delta|N}{\sum}}\frac{\text{gcd}(\delta,d)^{2}}	
		{\delta} r_{\delta} > 0$ holds 
		
then $f(z)\in \EuScript{S}_{k}(\Gamma_{0}(N))$.
\end{theorem}

\subsection{Evaluating $W_{(\alpha, \beta)}(n)$}
\label{convolutionSumsEqns}

Let $\alpha,\beta\in\mathbb{N}$ be such that $\alpha\leq\beta$. 
We let the convolution sum, 
 $W_{(\alpha,\beta)}(n)$, be defined as in \autoref{def-convolution_sum}.

Following the observation by A.~Alaca et al.\ \cite{alaca_alaca_williams2006}, 
we assume that $\text{gcd}(\alpha,\beta)=1$.  
Suppose that $q\in\mathbb{C}$ is such that $|q|<1$. 
 We define the Eisenstein series $L(q)$ and $M(q)$ by 
\begin{align}  
L(q) = E_{2}(q) = 1-24\,\sum_{n=1}^{\infty}\sigma(n)q^{n}, 
\label{evalConvolClass-eqn-3} \\
M(q) = E_{4}(q) = 1 + 240\,\sum_{n=1}^{\infty}\sigma_{3}(n)q^{n}. 
\label{evalConvolClass-eqn-4}
\end{align}
The following two results whose proofs are given by \aA\ et al.\ 
\cite{alaca_alaca_ntienjem2016} are essential for the sequel of this work 
\begin{lemma}  \label{evalConvolClass-lema-1}
Let $\alpha, \beta \in \mathbb{N}$. Then 
\begin{equation*}
( \alpha\, L(q^{\alpha}) - \beta\, L(q^{\beta}) )^{2}\in
\M_{4}(\Gamma_{0}(\alpha\beta)).
\end{equation*}
\end{lemma}
\begin{theorem} \label{convolutionSum_a_b}
Let $\alpha,\beta\in\mathbb{N}$ be such that
$\alpha$ and $\beta$ are relatively prime and $\alpha < \beta$. 
Then
\begin{align}
 ( \alpha\, L(q^{\alpha}) - \beta\, L(q^{\beta}) )^{2}  =  & 
 (\alpha - \beta)^{2} 
    + \sum_{n=1}^{\infty}\biggl(\ 240\,\alpha^{2}\,\sigma_{3}
    (\frac{n}{\alpha}) 
    + 240\,\beta^{2}\,\sigma_{3}(\frac{n}{\beta})  \notag \\  & 
    + 48\,\alpha\,(\beta-6n)\,\sigma(\frac{n}{\alpha}) 
    + 48\,\beta\,(\alpha-6n)\,\sigma(\frac{n}{\beta}) \notag \\  & 
    - 1152\,\alpha\beta\,W_{(\alpha,\beta)}(n)\,\biggr)q^{n}. 
    \label{evalConvolClass-eqn-11}
\end{align} 
\end{theorem}


\section{Evaluation of the convolution sums $W_{(\alpha,\beta)}(n)$, where 
$\alpha\beta=44,52$}
\label{convolution_44_52}

We give explicit formulae for the convolution sums 
$W_{(1,44)}(n)$, $W_{(4,11)}(n)$, $W_{(1,52}(n)$ and $W_{(4,13}(n)$. 

\subsection{Bases for $\E_{4}(\Gamma_{0}(\alpha\beta))$ and 
  $\EuScript{S}_{4}(\Gamma_{0}(\alpha\beta))$ with $\alpha\beta=44,52$}  
\label{convolution_44_52-gen}

We apply the dimension formulae for the space of Eisenstein forms and 
the space of cusp forms in T.~Miyake's book  
\cite[Thrm 2.5.2,~p.~60]{miyake1989} or W.~A.~Stein's book  
\cite[Prop.\ 6.1, p.\ 91]{wstein} to compute    
$\text{dim}(\E_{4}(\Gamma_{0}(44)))=\text{dim}(\E_{4}(\Gamma_{0}(52)))=6$, 
$\text{dim}(\EuScript{S}_{4}(\Gamma_{0}(44))=15$ and 
$\text{dim}(\EuScript{S}_{4}(\Gamma_{0}(52))=18$.  

Let $D(44)=\{1,2,4,11,22,44\}$ and $D(52)=\{1,2,4,13,26,52\}$ be the sets 
of positive divisors of $44$ and $52$, respectively.

We apply \autoref{ligozat_theorem} $(i)-(v')$ to determine as many elements 
of $\EuScript{S}_{4}(\Gamma_{0}(44))$ and 
$\EuScript{S}_{4}(\Gamma_{0}(52))$ as possible. From these elements 
we then determine the basis elements.
\begin{theorem} \label{basisCusp_44_52}
\begin{enumerate}
\item[\textbf{(a)}] The sets $\EuScript{B}_{E,44}=\{\,M(q^{t})\,\mid ~t\in D(44)\,\}$ 
and $\EuScript{B}_{E,52}=\{\, M(q^{t})\,\mid ~ t\in D(52)\,\}$
are bases of $\E_{4}(\Gamma_{0}(44))$ and 
$\E_{4}(\Gamma_{0}(52))$, respectively. 
\item[\textbf{(b)}] Let $1\leq i\leq 15$ and $1\leq j\leq 18$ 
be positive integers. 

Let $\delta_{1}\in D(44)$ and 
$(r(i,\delta_{1}))_{i,\delta_{1}}$ be the 
\autoref{convolutionSums-4_11-table} of the powers of $\eta(\delta_{1} z)$. 

Let $\delta_{2}\in D(52)$ and 
$(r(j,\delta_{2}))_{j,\delta_{2}}$ be the 
\autoref{convolutionSums-4_13-table} of the powers of $\eta(\delta_{2} z)$. 

Let furthermore  
$A_{i}(q)=\underset{\delta_{1}|44}{\prod}\eta^{r(i,\delta_{1})}(\delta_{1}
z)$ and   
$B_{j}(q)=\underset{\delta_{2}|52}{\prod}\eta^{r(j,\delta_{2})}(\delta_{2}
z)$
be selected elements of 
$\EuScript{S}_{4}(\Gamma_{0}(44))$ and $\EuScript{S}_{4}(\Gamma_{0}(52))$, 
respectively. 

Then the sets 
$\EuScript{B}_{S,44}=\{\,A_{i}(q)\,\mid ~ 1\leq i\leq 15\,\}$ and 
$\EuScript{B}_{S,52}=\{\,B_{j}(q)\,\mid ~ 1\leq j\leq 18\,\}$ 
are bases of $\EuScript{S}_{4}(\Gamma_{0}(44))$ and 
$\EuScript{S}_{4}(\Gamma_{0}(52))$, repectively.
\item[\textbf{(c)}] The sets
$\EuScript{B}_{M,44}=\EuScript{B}_{E,44}\cup\EuScript{B}_{S,44}$ and 
$\EuScript{B}_{M,52}=\EuScript{B}_{E,52}\cup\EuScript{B}_{S,52}$
constitute bases of $\M_{4}(\Gamma_{0}(44))$ and 
$\M_{4}(\Gamma_{0}(52))$, respectively.
\end{enumerate}
\end{theorem}
For $1\leq i\leq 15$ and $1\leq j\leq 18$ the eta quotients $A_{i}(q)$ 
and $B_{j}(q)$ can be expressed in the form 
$\underset{n=1}{\overset{\infty}{\sum}}a_{i}(n)q^{n}$ 
and $\underset{n=1}{\overset{\infty}{\sum}}b_{j}(n)q^{n}$, respectively.

\begin{proof} We only prove the case $\alpha\beta=44$. The case 
$\alpha\beta=52$ is proved similarly. 
\begin{enumerate}
\item[(a)] When we apply Theorem 5.8 in Section 5.3 of \waS\ \cite[p.~86]{wstein}, 
it follows that $M(q^{t})$ belongs to $\M_{4}(\Gamma_{0}(t))$ for each $t\in D(44)$. 
Since $\E_{4}(\Gamma_{0}(44))$ has a finite dimension, it is sufficient 
to show that the set of $M(q^{t})$ such that $t\in D(44)$ is linearly independent. 
Suppose that $x_{t}\in\mathbb{C}$ with $t|44$. Then  
\begin{equation*}
\underset{t|44}{\sum}x_{t}\,M(q^{t})=\underset{t|44}{\sum}x_{t}+
240\,\underset{n\geq 1}{\sum}\biggl(\underset{t|44}{\sum}x_{t}\sigma_{3}(\frac{n}{t})
\biggr)q^{n}=0. 
\end{equation*}
We compare the coefficients of $q^{n}$ for $n\in D(44)$ 
to obtain the following homogeneous system of $6$ equations 
in $6$ unknowns:
\begin{equation*}
\underset{u|44}{\sum}\sigma_{3}(\frac{t}{u})x_{u}=0,\qquad t\in D(44).
\end{equation*}
The matrix of this homogeneous system of equations is triangular with 
positive integer values, $1$, on the diagonal. 
Hence, the solution is $x_{t}=0$ for all $t\in D(44)$. 
Therefore, the set 
$\EuScript{B}_{E}$ is linearly independent and hence  
is a basis of $\E_{4}(\Gamma_{0}(44))$.

\item[(b)] As mentioned above, the $A_{i}(q)$ with $1\leq i\leq 15$ are obtained 
from an exhaustive search using \autoref{ligozat_theorem} $(i)-(v')$. 
Hence, each 
$A_{i}(q)$ is in the space $\EuScript{S}_{4}(\Gamma_{0}(44))$.

Since the dimension of $\EuScript{S}_{4}(\Gamma_{0}(44))$ is $15$, it is 
sufficient to show that the set $\{\,A_{i}(q)\mid 1\leq i\leq 15\}$ 
is linearly independent. 
 For that suppose that $x_{i}\in\mathbb{C}$ and 
$\underset{i=1}{\overset{15}{\sum}}x_{i}\,A_{i}(q)=0$. Then   
\begin{equation*}
\underset{i=1}{\overset{15}{\sum}}x_{i}\,A_{i}(q)
= \underset{n=1}{\overset{\infty}{\sum}}(\,\underset{i=1}{\overset{15}{\sum}}x_{i}\,a_{i}(n)\,)q^{n} = 0
\end{equation*}
which gives the following homogeneous system of equations in $15$ unknowns  
\begin{equation}  \label{basis-cusp-eqn-sol}
\underset{i=1}{\overset{15}{\sum}}\,a_{i}(n)\,x_{i}= 0,\qquad 
1\leq n\leq 15.
\end{equation}
A computation using a software package for (symbolic) scientific 
computation shows that 
the determinant of the matrix of this homogeneous system of equations is 
non-zero. So, $x_{i}=0$ for all $1\leq i\leq 15$. Hence, the set 
$\{\, A_{i}(q)\mid 1\leq i\leq 15\,\}$ is linearly 
independent and therefore a basis of $\EuScript{S}_{4}(\Gamma_{0}(44))$.

\item[(c)] Since $\M_{4}(\Gamma_{0}(44))=\E_{4}(\Gamma_{0}(44))\oplus 
\EuScript{S}_{4}(\Gamma_{0}(44))$, the result follows from (a) and (b).
\end{enumerate}
\end{proof}
We observe that the basis elements 
\begin{enumerate}
	\item[(o1)] $A_{i}(q)$, $1\leq i\leq 5$, come from $\EuScript{S}_{4}(\Gamma_{0}(22))$  
	which is the space of cusp forms necessary for the evaluation of the convolution 
	sums $W_{22}(n)$ and $W_{(2,11)}(n)$ given by \aA\ et al.\ \cite{alaca_alaca_ntienjem2016}. The element $A_{2}(q)$ is inherited from 
	$\EuScript{S}_{4}(\Gamma_{0}(11))$ which is part of $\M_{4}(\Gamma_{0}(11))$; 
	the convolution sum $W_{11}(n)$ is evatuated by \eR\ \cite{royer}.
	\item[(o2)] $A_{2i}(q)=A_{i}(q^{2})$, for $i = 2,3,4,5$. Therefore, 
		$a_{2i}(n)=a_{i}(\frac{n}{2})$, for $i = 2,3,4,5$.
	\item[(o3)] $B_{j}(q)$, $1\leq j\leq 7$, $B_{15}(q)$ and $B_{17}(q)$ are imported from 
	$\EuScript{S}_{4}(\Gamma_{0}(26))$ which is the space of cusp forms required for  
	the evaluation of the convolution sums 
	 $W_{26}(n)$ and $W_{(2,13)}(n)$ given by  \aA\ et al.\ \cite{alaca_alaca_ntienjem2016}.
	\item[(o4)] $B_{2j}(q)=B_{j}(q^{2})$, for $4\leq j\leq 7$, $B_{16}(q)=B_{15}(q^{2})$ and 
	$B_{18}(q)=B_{17}(q^{2})$. Consequently, 
		$b_{2j}(n)=b_{j}(\frac{n}{2})$, for $4\leq j\leq 7$, 
		$b_{16}(n)=b_{15}(\frac{n}{2})$ and $b_{18}(n)=b_{17}(\frac{n}{2})$.
\end{enumerate}
The above observation is based on the fact that 
\begin{align*}
\M_{4}(\Gamma_{0}(11)) \subset \M_{4}(\Gamma_{0}(22)) \subset \M_{4}(\Gamma_{0}(44)) 
\quad \text{and} \\
\M_{4}(\Gamma_{0}(13)) \subset \M_{4}(\Gamma_{0}(26)) \subset \M_{4}(\Gamma_{0}(52)).
\end{align*}
As mentioned in (o1) above, the eta quotient $A_{2}(q)$ is a basis element of 
$\EuScript{S}_{4}(\Gamma_{0}(11))$. Hence, basis elements of 
$\EuScript{S}_{4}(\Gamma_{0}(11))$ can be determined using \autoref{ligozat_theorem}. 
There is no basis element of $\EuScript{S}_{4}(\Gamma_{0}(13))$ in the space of 
cusp forms $\EuScript{S}_{4}(\Gamma_{0}(26))$; this is an indication that there is no 
basis element of $\EuScript{S}_{4}(\Gamma_{0}(13))$ that can be determined using 
\autoref{ligozat_theorem}.

\subsection{Evaluation of $W_{(\alpha,\beta)}(n)$ when $\alpha\beta=44,52$} 
\label{convolSum-w_44_52}

\begin{lemma} \label{lema-w_44_52}
We have  
\begin{multline}
( L(q) - 44\, L(q^{44}))^{2}  
 = 1849 + \sum_{n=1}^{\infty}\biggl(\, 
   \frac{124464}{61}\,\sigma_{3}(n)
- \frac{577662336}{40565}\,\sigma_{3}(\frac{n}{2})  \\  
+  \frac{68986368}{5795}\,\sigma_{3}(\frac{n}{4})   
  - \frac{174240}{61}\,\sigma_{3}(\frac{n}{11}) 
+  \frac{62064288}{5795}\,\sigma_{3}(\frac{n}{22}) 
+ \frac{2525690112}{5795}\,\sigma_{3}(\frac{n}{44}) \\
 +   \frac{1440}{61}\,a_{1}(n)
 - \frac{82927872}{5795}\,a_{2}(n)    
- \frac{887345568}{5795}\,a_{3}(n)
- \frac{1676429568}{5795}\,a_{4}(n)   \\
- \frac{2804007168}{5795}\,a_{5}(n)   
+ \frac{3753380736}{5795}\,a_{6}(n)
 - \frac{13356288}{19}\,a_{7}(n)    
+ \frac{4226609664}{5795}\,a_{8}(n)    \\ 
  - \frac{633600}{19}\,a_{9}(n)  
- \frac{527332608}{1159}\,a_{10}(n) 
+   \frac{7679232}{19}\,a_{11}(n)     
 - \frac{15231744}{95}\,a_{12}(n)   \\
- \frac{131079168}{95}\,a_{13}(n) 
 +  \frac{317952}{19}\,a_{14}(n)     
 - \frac{12595968}{95}\,a_{15}(n)
 \biggr)q^{n}, \label{convolSum-eqn-1_44}  
\end{multline}
\begin{multline}
( 4\,L(q^{4}) - 11\, L(q^{11}))^{2}  
 = 49 + \sum_{n=1}^{\infty}\biggl(\, 
 - \frac{110880}{61}\,\sigma_{3}(n)
 + \frac{80121888}{5795}\,\sigma_{3}(\frac{n}{2}) \\  
 - \frac{48338688}{5795}\,\sigma_{3}(\frac{n}{4})   
 +  \frac{1817904}{61}\,\sigma_{3}(\frac{n}{11}) 
 - \frac{98480448}{5795}\,\sigma_{3}(\frac{n}{22})  
 - \frac{27320832}{5795}\,\sigma_{3}(\frac{n}{44}) \\
+   \frac{110880}{61}\,a_{1}(n)
+  \frac{174857472}{5795}\,a_{2}(n)   
+ \frac{1169427168}{5795}\,a_{3}(n)
+ \frac{2114189568}{5795}\,a_{4}(n)  \\
+ \frac{3025513728}{5795}\,a_{5}(n)  
- \frac{3511080576}{5795}\,a_{6}(n)
+  \frac{13318272}{19}\,a_{7}(n)   
- \frac{3641762304}{5795}\,a_{8}(n)   \\ 
+   \frac{633600}{19}\,a_{9}(n)
+  \frac{663913728}{1159}\,a_{10}(n)
 - \frac{7679232}{19}\,a_{11}(n)    
+  \frac{15231744}{95}\,a_{12}(n)    \\
+  \frac{131079168}{95}\,a_{13}(n)
  - \frac{317952}{19}\,a_{14}(n)     
+  \frac{12595968}{95}\,a_{15}(n)
\,\biggr)q^{n},
\label{convolSum-eqn-4_11} 
\end{multline}
\begin{multline}
( L(q) - 52\, L(q^{52}))^{2}   
 = 2601 + \sum_{n=1}^{\infty}\biggl(\, 
    \frac{6109008}{1243}\,\sigma_{3}(n)
 - \frac{456504084816}{6064597}\,\sigma_{3}(\frac{n}{2})  \\     
 +    \frac{254592}{41}\,\sigma_{3}(\frac{n}{4})   
   - \frac{7361952}{1243}\,\sigma_{3}(\frac{n}{13}) 
- \frac{4829528827344}{6064597}\,\sigma_{3}(\frac{n}{26})   \\
 +   \frac{434738304}{41}\,\sigma_{3}(\frac{n}{52})   
   - \frac{3066144}{1243}\,b_{1}(n)   
+  \frac{498157179048}{6064597}\,b_{2}(n)  \\     
+  \frac{927327070704}{6064597}\,b_{3}(n)  
 - \frac{442577500560}{6064597}\,b_{4}(n)   
- \frac{8530413669648}{6064597}\,b_{5}(n)   \\
- \frac{10161699732288}{6064597}\,b_{6}(n)      
 - \frac{10388366352}{1243}\,b_{7}(n)   
  +   \frac{1040832}{41}\,b_{8}(n)   \\ 
 +    7488\,b_{9}(n)   
+  \frac{329100929664}{147917}\,b_{10}(n)   
 +    27456\,b_{11}(n)      \\
- \frac{15249288510144}{6064597}\,b_{12}(n)   
 +    17472\,b_{13}(n)    
 +   \frac{47009664}{41}\,b_{14}(n)    \\
 - \frac{25166713896}{551327}\,b_{15}(n)    
+  \frac{4167031826880}{6064597}\,b_{16}(n)     
 - \frac{126425023920}{6064597}\,b_{17}(n)  \\  
  +   \frac{868608}{41}\,b_{18}(n)
\, \biggr)q^{n}, \label{convolSum-eqn-1_52}  
\end{multline}
\begin{multline}
(4\, L(q^{4}) - 13\, L(q^{13}))^{2} 
 = 81 + \sum_{n=1}^{\infty}\biggl(\, 
    \frac{3066144}{1243}\,\sigma_{3}(n)
- \frac{240061230672}{6064597}\,\sigma_{3}(\frac{n}{2})   \\  
  +   \frac{139392}{41}\,\sigma_{3}(\frac{n}{4}) 
  +  \frac{45798672}{1243}\,\sigma_{3}(\frac{n}{13}) 
 - \frac{53922031824}{6064597}\,\sigma_{3}(\frac{n}{26})  \\
 +   \frac{20290176}{41}\,\sigma_{3}(\frac{n}{52})  
   - \frac{3066144}{1243}\,b_{1}(n)  
+  \frac{212735819880}{6064597}\,b_{2}(n)   \\ 
+  \frac{251848851024}{6064597}\,b_{3}(n) 
- \frac{400561037808}{6064597}\,b_{4}(n)  
- \frac{5152459820400}{6064597}\,b_{5}(n)   \\ 
- \frac{5408748312192}{6064597}\,b_{6}(n)    
 - \frac{5489355312}{1243}\,b_{7}(n)  
 +    \frac{150336}{41}\,b_{8}(n)   \\
    - 7488\,b_{9}(n)   
+  \frac{151016538432}{147917}\,b_{10}(n) 
    - 27456\,b_{11}(n)    \\ 
- \frac{8224832431680}{6064597}\,b_{12}(n)  
    - 17472\,b_{13}(n)   
   - \frac{544896}{41}\,b_{14}(n)   \\
 - \frac{11115614088}{551327}\,b_{15}(n)
+ \frac{2056953609600}{6064597}\,b_{16}(n)    
 - \frac{64745693328}{6064597}\,b_{17}(n)    \\
    - \frac{2304}{41}\,b_{18}(n)
\,\biggr)q^{n}.
\label{convolSum-eqn-4_13}
\end{multline}
\end{lemma}
\begin{proof} We just prove the case $( 4\,L(q^{4}) - 11\,L(q^{11}))^{2}$. 
The other cases are proved similarly.

From 
\hyperref[evalConvolClass-lema-1]{Lemma \ref*{evalConvolClass-lema-1}} it follows 
that $( 4\, L(q^{4}) - 11\, L(q^{11}) )^{2}\in
\M_{4}(\Gamma_{0}(44))$. 
Hence, by \autoref{basisCusp_44_52}\,(c), there exist 
$X_{\delta},Y_{j}\in\mathbb{C}, 
1\leq j\leq 15\text{ and } \delta\in D(44)$, such that  
\begin{align}
( 4\, L(q^{4}) - 11\, L(q^{11}) )^{2}   & =  
\sum_{\delta|44}X_{\delta}\,M(q^{\delta}) + \sum_{j=1}^{15}\,Y_{j}\,
                                        A_{j}(q)  \label{convolution_44_52-eqn-0a}
\\
 ~  & = \sum_{\delta|44}X_{\delta} + \sum_{n=1}^{\infty}\biggl(\, 
 240\,\sum_{\delta|44}\,\sigma_{3}(\frac{n}{\delta})\,X_{\delta} + 
 \sum_{j=1}^{m_{S}}\,a_{j}(n)\,Y_{j}\, \biggr)q^{n}. \notag
\end{align} 
We compare the right hand side of 
\autoref{convolution_44_52-eqn-0a} with that of   
\autoref{evalConvolClass-eqn-11} when we have set $(\alpha,\beta)=(4,11)$ to 
obtain 
\begin{align*}
\sum_{n=1}^{\infty}\biggl(\, 
 240\,\sum_{\delta|44}\,\sigma_{3}(\frac{n}{\delta})\,X_{\delta} + 
 \sum_{j=1}^{15}\,a_{j}(n)\,Y_{j}\, \biggr)q^{n} = \sum_{n=1}^{\infty}\biggl(\
 3840\,\sigma_{3}
    (\frac{n}{4}) 
    + 29040\,\sigma_{3}(\frac{n}{11}) \\
    + 192\,(11 - 6\,n)\,\sigma(\frac{n}{4})  
    + 528\,(4 - 6\,n)\,\sigma(\frac{n}{11})   
    - 50688\,W_{(4,11)}(n)\,\biggr)q^{n}. 
\end{align*}
We then take the coefficients of $q^{n}$ for which $n$ is in 
$$\{\,1,2,3,4,5,6,7,8,9,10,11,12,13,14,15,16,17,18,20,22,44\, \}.$$
to obtain  
a system of linear equations whose resolution using a software package for 
symbolic scientific computation yields the unique solution which determines 
the values of the unknowns $X_{\delta}$ for all $\delta\in D(44)$ and the 
values of the unkowns $Y_{j}$ for all $1\leq j\leq 15$. 
Therefore, we get the stated result.
\end{proof}
Our main result of this section will now be stated and proved. 
\begin{theorem} \label{convolSum-theor-w_44_52}
Let $n$ be a positive integer. Then 
\begin{align}
W_{(1,44)}(n)   = & 
   - \frac{13}{366}\,\sigma_{3}(n)
+  \frac{501443}{1784860}\,\sigma_{3}(\frac{n}{2}) 
  - \frac{1361}{5795}\,\sigma_{3}(\frac{n}{4}) 
  +  \frac{55}{976}\,\sigma_{3}(\frac{n}{11})   \notag \\ & 
  - \frac{19591}{92720}\,\sigma_{3}(\frac{n}{22}) 
 +  \frac{9878}{17385}\,\sigma_{3}(\frac{n}{44}) 
 + (\frac{1}{24}-\frac{1}{176}n)\sigma(n)  \notag \\ & 
 + (\frac{1}{24}-\frac{1}{4}n)\sigma(\frac{n}{44})  
   - \frac{5}{10736}\,a_{1}(n)
 +  \frac{35993}{127490}\,a_{2}(n)  \notag \\ & 
+  \frac{3081061}{1019920}\,a_{3}(n)  
+   \frac{66147}{11590}\,a_{4}(n)
 + \frac{1217017}{127490}\,a_{5}(n)   
 - \frac{3258143}{254980}\,a_{6}(n) \notag \\ & 
  +  \frac{527}{38}\,a_{7}(n)   
 - \frac{917233}{63745}\,a_{8}(n)
  +   \frac{25}{38}\,a_{9}(n)
 +  \frac{20807}{2318}\,a_{10}(n)  \notag \\ &  
   - \frac{303}{38}\,a_{11}(n)   
  +  \frac{601}{190}\,a_{12}(n)
  +  \frac{2586}{95}\,a_{13}(n) 
   - \frac{69}{209}\,a_{14}(n)  
  +  \frac{497}{190}\,a_{15}(n), 
\label{convolSum-theor-w_1_44}  
\end{align} 
\begin{align}
W_{(4,11)}(n)  = & 
    \frac{35}{976}\,\sigma_{3}(n)
  - \frac{25291}{92720}\,\sigma_{3}(\frac{n}{2})  
 +  \frac{4178}{17385}\,\sigma_{3}(\frac{n}{4}) 
   - \frac{11}{732}\,\sigma_{3}(\frac{n}{11})   \notag \\ & 
  + \frac{15543}{46360}\,\sigma_{3}(\frac{n}{22}) 
  +  \frac{539}{5795}\,\sigma_{3}(\frac{n}{44}) 
 + (\frac{1}{24}-\frac{1}{44}n)\sigma(\frac{n}{4})  \notag \\ & 
 + (\frac{1}{24}-\frac{1}{16}n)\sigma(\frac{n}{11}) 
   - \frac{35}{976}\,a_{1}(n)
  - \frac{75893}{127490}\,a_{2}(n)   \notag \\ & 
 - \frac{4060511}{1019920}\,a_{3}(n)
 - \frac{917617}{127490}\,a_{4}(n)
 - \frac{1313157}{127490}\,a_{5}(n)  
 + \frac{3047813}{254980}\,a_{6}(n)   \notag \\ & 
  - \frac{1051}{76}\,a_{7}(n)
 +  \frac{790313}{63745}\,a_{8}(n)
   - \frac{25}{38}\,a_{9}(n)
 - \frac{288157}{25498}\,a_{10}(n)   \notag \\ &  
  +  \frac{303}{38}\,a_{11}(n)
   - \frac{601}{190}\,a_{12}(n)
  - \frac{2586}{95}\,a_{13}(n)
  +  \frac{69}{209}\,a_{14}(n)  
   - \frac{497}{190}\,a_{15}(n),  
\label{convolSum-theor-w_4_11}  
 \end{align}
\begin{align}
W_{(1,52)}(n)  = & 
     - \frac{97}{1243}\,\sigma_{3}(n)
  + \frac{731577059}{582201312}\,\sigma_{3}(\frac{n}{2})    
     - \frac{17}{164}\,\sigma_{3}(\frac{n}{4}) 
   +  \frac{5899}{59664}\,\sigma_{3}(\frac{n}{13})     \notag \\ & 
  + \frac{7739629531}{582201312}\,\sigma_{3}(\frac{n}{26})   
    - \frac{81757}{492}\,\sigma_{3}(\frac{n}{52})
 + (\frac{1}{24}-\frac{1}{208}n)\sigma(n)  \notag \\ & 
 + (\frac{1}{24}-\frac{1}{4}n)\sigma(\frac{n}{52})   
   +  \frac{31939}{775632}\,b_{1}(n)
 - \frac{6918849709}{5045744704}\,b_{2}(n)  \notag \\ & 
 - \frac{19319313973}{7568617056}\,b_{3}(n)
  + \frac{236419605}{194067104}\,b_{4}(n)     
  + \frac{4556844909}{194067104}\,b_{5}(n)   \notag \\ & 
  + \frac{1357064601}{48516776}\,b_{6}(n)
  +  \frac{5549341}{39776}\,b_{7}(n)
     - \frac{139}{328}\,b_{8}(n)    \notag \\ &  
     - \frac{1}{8}\,b_{9}(n)
  - \frac{65925667}{1775004}\,b_{10}(n)   
     - \frac{11}{24}\,b_{11}(n)    
  + \frac{2036496863}{48516776}\,b_{12}(n)   \notag \\ & 
      - \frac{7}{24}\,b_{13}(n)
    - \frac{3139}{164}\,b_{14}(n)    
  + \frac{349537693}{458704064}\,b_{15}(n)  \notag \\ & 
  - \frac{556494635}{48516776}\,b_{16}(n)  
  + \frac{67534735}{194067104}\,b_{17}(n)    
     - \frac{29}{82}\,b_{18}(n),
\label{convolSum-theor-w_1_52}  
\end{align} 
\begin{align}
W_{(4,13)}(n)  = & 
    - \frac{31939}{775632}\,\sigma_{3}(n)
  + \frac{5001275639}{7568617056}\,\sigma_{3}(\frac{n}{2}) 
     +  \frac{47}{6396}\,\sigma_{3}(\frac{n}{4}) 
   +  \frac{24049}{387816}\,\sigma_{3}(\frac{n}{13}) \notag \\ & 
  + \frac{1123375663}{7568617056}\,\sigma_{3}(\frac{n}{26}) 
    - \frac{17613}{2132}\,\sigma_{3}(\frac{n}{52}) 
 + (\frac{1}{24}-\frac{1}{52}n)\sigma(\frac{n}{4}) \notag \\ & 
 + (\frac{1}{24}-\frac{1}{16}n)\sigma(\frac{n}{13})  
   +  \frac{31939}{775632}\,b_{1}(n)
  - \frac{2954664165}{5045744704}\,b_{2}(n)  \notag \\ & 
  - \frac{5246851063}{7568617056}\,b_{3}(n)
  + \frac{8345021621}{7568617056}\,b_{4}(n)   
  + \frac{35780970975}{2522872352}\,b_{5}(n)  \notag \\ & 
  + \frac{4695094021}{315359044}\,b_{6}(n)
  +  \frac{38120523}{517088}\,b_{7}(n)
     - \frac{261}{4264}\,b_{8}(n)    \notag \\ & 
    +   \frac{1}{8}\,b_{9}(n)
  - \frac{786544471}{46150104}\,b_{10}(n)   
   +    \frac{11}{24}\,b_{11}(n)   
  + \frac{42837668915}{1892154264}\,b_{12}(n)   \notag \\ & 
    +  \frac{7}{24}\,b_{13}(n)
   +   \frac{473}{2132}\,b_{14}(n) 
  +  \frac{154383529}{458704064}\,b_{15}(n)   \notag \\ & 
  - \frac{5356650025}{946077132}\,b_{16}(n)   
  + \frac{1348868611}{7568617056}\,b_{17}(n)  
  +  \frac{1}{1066}\,b_{18}(n).
\label{convolSum-theor-w_4_13}  
 \end{align}
\end{theorem}
\begin{proof} We prove the case $W_{(4,13)}(n)$ as the other cases 
are proved similarly.

We compare the right hand side of  
\autoref{convolSum-eqn-4_13} with that of 
\autoref{evalConvolClass-eqn-11} when we have set $(\alpha,\beta)=(4,13)$, namely  
 \begin{multline*} 
 \sum_{n=1}^{\infty}\biggl(\ 3840\,\sigma_{3}(\frac{n}{4}) 
    + 40560\,\sigma_{3}(\frac{n}{13}) 
    + 192\,(13 - 6\,n)\,\sigma(\frac{n}{4})   \\
    + 624\,(4 - 6\,n)\,\sigma(\frac{n}{13}) 
    - 59904\,W_{(4,13)}(n)\,\biggr)q^{n} =  \\
 \sum_{n=1}^{\infty}\biggl(\, 
    \frac{3066144}{1243}\,\sigma_{3}(n)
- \frac{240061230672}{6064597}\,\sigma_{3}(\frac{n}{2})   
  +   \frac{139392}{41}\,\sigma_{3}(\frac{n}{4})  \\
  +  \frac{45798672}{1243}\,\sigma_{3}(\frac{n}{13}) 
 - \frac{53922031824}{6064597}\,\sigma_{3}(\frac{n}{26}) 
 +   \frac{20290176}{41}\,\sigma_{3}(\frac{n}{52})  \\
   - \frac{3066144}{1243}\,b_{1}(n)  
+  \frac{212735819880}{6064597}\,b_{2}(n)
+  \frac{251848851024}{6064597}\,b_{3}(n)  
- \frac{400561037808}{6064597}\,b_{4}(n)  \\
- \frac{5152459820400}{6064597}\,b_{5}(n)
- \frac{5408748312192}{6064597}\,b_{6}(n)  
 - \frac{5489355312}{1243}\,b_{7}(n)  \\
 +    \frac{150336}{41}\,b_{8}(n)
    - 7488\,b_{9}(n)   
+  \frac{151016538432}{147917}\,b_{10}(n) 
    - 27456\,b_{11}(n)   \\
- \frac{8224832431680}{6064597}\,b_{12}(n)  
    - 17472\,b_{13}(n)  
   - \frac{544896}{41}\,b_{14}(n)  
 - \frac{11115614088}{551327}\,b_{15}(n)  \\
+ \frac{2056953609600}{6064597}\,b_{16}(n)  
 - \frac{64745693328}{6064597}\,b_{17}(n)  
    - \frac{2304}{41}\,b_{18}(n)
\,\biggr)q^{n} 
\end{multline*}
We obtain the stated result when we solve for $W_{(4,13)}(n)$.
\end{proof}


\section{Number of Representations of a positive Integer  $n$  by the Octonary 
Quadratic Form using  $W_{(\alpha,\beta)}(n)$ when $\alpha\beta=44,52$
}
\label{representations_44_52}

Let $n\in\mathbb{N}_{0}$ and then assune that $r_{4}(n)$ denote the number of representations
of $n$ by the quaternary quadratic form  $x_{1}^{2} +x_{2}^{2}+x_{3}^{2} +
x_{4}^{2}$ which is defined by  
\begin{equation*}
r_{4}(n)=\text{card}(\{(x_{1},x_{2},x_{3},x_{4})\in\mathbb{Z}^{4}~|~ n = x_{1}^{2} 
+x_{2}^{2} + x_{3}^{2} + x_{4}^{2}\}).
\end{equation*}
Obviously $r_{4}(0) = 1$. The Jacobi's identity 
\begin{equation}
\forall n\in\mathbb{N}\qquad  r_{4}(n) = 8\sigma(n) - 32\sigma(\frac{n}{4}). 
\label{representations-eqn-4-1}
\end{equation}
is proved in \ksW ' book \cite[Thrm 9.5, p.\ 83]{williams2011}; it will be 
very useful in the following.  

Let $a,b\in\mathbb{N}$ and let $N_{(a,b)}(n)$ denote the number of 
representations of $n$ by the octonary quadratic form 
\begin{equation*}
a\,(x_{1}^{2} +x_{2}^{2} + x_{3}^{2} + x_{4}^{2})
+ b\,(x_{5}^{2} + x_{6}^{2} + x_{7}^{2} + x_{8}^{2})
\end{equation*}
which is defined by 
\begin{align*}
N_{(a,b)}(n) 
=\text{card}
(\{(x_{1},x_{2},x_{3},x_{4},x_{5},x_{6},x_{7},x_{8})\in\mathbb{Z}^{8}~|~
n = a\,( x_{1}^{2} +x_{2}^{2}  \\
    + x_{3}^{2} + x_{4}^{2} ) + 
b\,( x_{5}^{2} +x_{6}^{2} + x_{7}^{2} + x_{8}^{2}) \}).
\end{align*}
We then deduce the following result.
\begin{theorem} \label{representations-thrm_3_4}
Let $n\in\mathbb{N}$ and $(a,b)=(1,11),(1,13)$. Then  
\begin{align*}
N_{(1,11)}(n)  = & 
8\sigma(n) - 32\sigma(\frac{n}{4}) + 8\sigma(\frac{n}{11}) -
32\sigma(\frac{n}{44}) \\ &
 + 64\, W_{(1,11)}(n) + 1024\, W_{(1,11)}(\frac{n}{4}) 
 - 256\, \biggl( W_{(4,11)}(n) + W_{(1,44)}(n) \biggr), \\ 
N_{(1,13)}(n)  = & 
8\sigma(n) - 32\sigma(\frac{n}{4}) + 8\sigma(\frac{n}{13}) -
32\sigma(\frac{n}{52}) \\ &
 + 64\, W_{(1,13)}(n) + 1024\, W_{(1,13)}(\frac{n}{4}) 
 - 256\, \biggl( W_{(4,13)}(n) + W_{(1,52)}(n) \biggr). 
\end{align*}
\end{theorem}
\begin{proof} We only prove $N_{(1,11)}(n)$ since that for $N_{(1,13)}(n)$ 
is done similarly. 

From the definition of $N_{(1,11)}(n)$ it follows that 
\begin{align*}
N_{(1,11)}(n)  =  
\sum_{\substack{
{(l,m)\in\mathbb{N}_{0}^{2}} \\ {l+11\,m=n}
 }}r_{4}(l)r_{4}(m)  
   =  r_{4}(n)r_{4}(0) + r_{4}(0)r_{4}(\frac{n}{11})  
   + \sum_{\substack{
{(l,m)\in\mathbb{N}^{2}} \\ {l+11\,m=n}
 }}r_{4}(l)r_{4}(m).
\end{align*}
We apply \autoref{representations-eqn-4-1} to derive 
\begin{align*}
N_{(1,11)}(n)  = & 
8\sigma(n) - 32\sigma(\frac{n}{4}) + 8\sigma(\frac{n}{11}) -
32\sigma(\frac{n}{52}) \\ &
   + \sum_{\substack{
{(l,m)\in\mathbb{N}^{2}} \\ {l+11\,m=n}
  }} (8\sigma(l) - 32\sigma(\frac{l}{4}))(8\sigma(m) - 32\sigma(\frac{m}{4})). 
\end{align*}
From this previous identity we observe that 
\begin{align*}
(8\sigma(l) - 32\sigma(\frac{l}{4}))(8\sigma(m) - 32\sigma(\frac{m}{4}))  = &
64\sigma(l)\sigma(m) - 256\sigma(\frac{l}{4})\sigma(m) \\ &
   - 256\sigma(l)\sigma(\frac{m}{4})  + 1024\sigma(\frac{l}
   {4})\sigma(\frac{m}{4}).
\end{align*}
\eR\ \cite[Thrm 1.3]{royer} has shown the evaluation of 
\begin{equation*}
W_{(1,11)}(n) = \sum_{\substack{
{(l,m)\in\mathbb{N}^{2}} \\ {l+11\,m=n}
 }}\sigma(l)\sigma(m).
\end{equation*}
When we assign $4l$ to $l$, then we infer  
\begin{equation*}
W_{(4,11)}(n) = \sum_{\substack{
{(l,m)\in\mathbb{N}^{2}} \\ {l+11\,m=n}
}}\sigma(\frac{l}{4})\sigma(m) 
 = \sum_{\substack{
{(l,m)\in\mathbb{N}^{2}} \\ {4\,l+11\,m=n}
 }}\sigma(l)\sigma(m).
\end{equation*}
The evaluation of $W_{(4,11)}(n)$ is given in
\autoref{convolSum-theor-w_4_11}. 
When we next assign $4m$ to $m$, we conclude  that 
\begin{equation*}
W_{(1,44)}(n) = \sum_{\substack{
{(l,m)\in\mathbb{N}^{2}} \\ {l+11\,m=n}
}}\sigma(l)\sigma(\frac{m}{4})  = \sum_{\substack{
{(l,m)\in\mathbb{N}^{2}} \\ {l+44\,m=n}
 }}\sigma(l)\sigma(m).
\end{equation*}
 The evaluation of $W_{(1,44)}(n)$ is  provided by 
\autoref{convolSum-theor-w_1_44}. 
When we simultaneously assign $4l$ to $l$ and $4m$ to $m$, we deduce that 
\begin{equation*}
\sum_{\substack{
{(l,m)\in\mathbb{N}^{2}} \\ {l+11\,m=n}
}}\sigma(\frac{l}{4})\sigma(\frac{m}{4}) 
= \sum_{\substack{
{(l,m)\in\mathbb{N}^{2}} \\ {l+11\,m=\frac{n}{4}}
 }}\sigma(l)\sigma(m)
 = W_{(1,11)}(\frac{n}{4}).
\end{equation*}
Again, \eR\ \cite[Thrm 1.3]{royer} has proved the evaluation 
of $W_{(1,11)}(n)$. 

We then bring these evaluations together to obtain the stated result 
for $N_{(1,11)}(n)$. 
\end{proof}


\begin{thebibliography}{10}
\bibitem{glaisher}
J.~W.~L. Glaisher.
\newblock On the square of the series in which the coefficients are the sums of
  the divisors of the exponents.
\newblock {\em Messenger Math}, 14:156--163, 1862.
\end{thebibliography}


\section*{Tables}

\begin{longtable}{|r|r|r|} \hline 
\textbf{$(\alpha,\beta)$}  &  \textbf{Authors}   &  \textbf{References}  \\ \hline
(1,1)  &  M.~Besge, J.~W.~L.~Glaisher, & ~ \\ 
 ~  & S.~Ramanujan  & \cite{besge,glaisher,ramanujan} \\ \hline
(1,2),(1,3),(1,4)  & J.~G.~Huard \& Z.~M.~Ou & ~ \\ 
 ~ &  \& B.~K.~Spearman \& K.~S.~Williams   & \cite{huardetal} \\ \hline
(1,5),(1,7)  & M.~Lemire \& K.~S.~Williams, & ~ \\ 
  ~  &  S.~Cooper \& P.~C.~Toh   & \cite{lemire_williams,cooper_toh} \\ \hline
(1,6),(2,3)  & S.~Alaca \& K.~S.~Williams   & \cite{alaca_williams} \\ \hline
(1,8), (1,9)  & K.~S.~Williams   & \cite{williams2006, williams2005} \\ \hline
(1,10), (1,11),(1,13), &  ~ & ~ \\ 
 (1,14)  & E.~Royer   & \cite{royer} \\ \hline
(1,12),(1,16),(1,18), &  ~ & ~ \\ 
(1,24),(2,9),(3,4), & A.~Alaca \& S.~Alaca \& K.~S.~Williams   &
\cite{alaca_alaca_williams2006,alaca_alaca_williams2007,alaca_alaca_williams2007a,alaca_alaca_williams2008}
\\ 
 (3,8)  &  ~ & ~ \\ \hline
(1,15),(3,5)  & B.~Ramakrishman \& B.~Sahu   & \cite{ramakrishnan_sahu} \\ \hline
(1,20),(2,5),(4,5)  & S.~Cooper \& D.~Ye   & \cite{cooper_ye2014} \\ \hline
(1,23)  & H.~H.~Chan \& S.~Cooper   & \cite{chan_cooper2008} \\ \hline
(1,25)  & E.~X.~W.~Xia \& X.~L.~Tian & ~ \\ 
 ~  &  \& O.~X.~M.~Yao   & \cite{xiaetal2014} \\ \hline
(1,27),(1,32)  & S.~Alaca \& Y.~Kesicio$\check{g}$lu   & \cite{alaca_kesicioglu2014} \\ \hline
(1,36),(4,9)  & D.~Ye   & \cite{ye2015} \\ \hline
(1,22),(1,26),(2,7), &  ~ & ~ \\   
(2,11),(2,13)  & \aA\ \& \sA\ \& \eN  & \cite{alaca_alaca_ntienjem2016} \\ \hline
\caption{Known convolution sums $W_{(\alpha, \beta)}(n)$} \label{introduction-table-1}
\end{longtable}

\begin{longtable}{|r|r|r|} \hline 
\textbf{$(a,b)$}  &  \textbf{Authors}   &  \textbf{References}  \\ \hline
(1,1),(1,3), & ~  & ~ \\
(1,9),(2,3) & \aA\ \& \sA\ \& \eN  & \cite{alaca_alaca_ntienjem2016} \\ \hline 
(1,2)  & K.~S.~Williams   & \cite{williams2006} \\ \hline
(1,4)  & A.~Alaca \& S.~Alaca \& K.~S.~Williams   & \cite{alaca_alaca_williams2007} \\ \hline
(1,5)  & S.~Cooper \& D.~Ye   & \cite{cooper_ye2014} \\ \hline
(1,6)  & B.~Ramakrishman \& B.~Sahu   & \cite{ramakrishnan_sahu}  \\ \hline
(1,8)  & S.~Alaca \& Y.~Kesicio$\check{g}$lu   & \cite{alaca_kesicioglu2014} \\ \hline
\caption{Known representations of $n$ by the form \autoref{introduction-eq-1}
} 
\label{introduction-table-rep2}
\end{longtable}

\begin{longtable}{|c|cccccc|} \hline
   & \textbf{1}  &  \textbf{2}  & \textbf{4} & \textbf{13} & \textbf{26} & \textbf{52}  \\ \hline
 \textbf{1}  & 1 & 5 & 0 & 3 & -1 & 0  \\ \hline 
 \textbf{2}  & 3 & 3 & 0 & 1 & 1 & 0  \\ \hline 
 \textbf{3}  & 1 & 3 & 0 & 3 & 1 & 0  \\ \hline 
 \textbf{4}  & 3 & 1 & 0 & 1 & 3 & 0  \\ \hline 
 \textbf{5}  & 1 & 1 & 0 & 3 & 3 & 0  \\ \hline 
 \textbf{6}  & 3 & -1 & 0 & 1 & 5 & 0  \\ \hline 
 \textbf{7}  & 1 & -1 & 0 & 3 & 5 & 0  \\ \hline 
 \textbf{8}  & 0 & 3 & 1 & 0 & 1 & 3  \\ \hline 
 \textbf{9}  & 2 & 1 & 1 & -2 & 3 & 3  \\ \hline 
 \textbf{10}  & 0 & 1 & 1 & 0 & 3 & 3  \\ \hline 
 \textbf{11}  & 2 & -1 & 1 & -2 & 5 & 3  \\ \hline 
 \textbf{12}  & 0 & 3 & -1 & 0 & 1 & 5  \\ \hline 
 \textbf{13}  & 2 & 1 & -1 & -2 & 3 & 5  \\ \hline 
 \textbf{14}  & 0 & 1 & -1 & 0 & 3 & 5  \\ \hline 
 \textbf{15}  & -1 & 5 & 0 & 5 & -1 & 0  \\ \hline 
 \textbf{16}  & 0 & -1 & 5 & 0 & 5 & -1  \\ \hline 
 \textbf{17}  & 7 & -3 & 0 & -3 & 7 & 0  \\ \hline 
 \textbf{18}  & 0 & 7 & -3 & 0 & -3 & 7  \\ \hline
\caption{Power of $\eta$-functions being basis elements for $S_{4}(\Gamma_{0}(52))$}
\label{convolutionSums-4_13-table}
\end{longtable}

\begin{longtable}{|c|cccccc|} \hline
   & \textbf{1}  &  \textbf{2}  & \textbf{4} & \textbf{11} & \textbf{22} & \textbf{44}  \\ \hline
\textbf{1} & 6  & -2  & 0  & 6  & -2  & 0   \\ \hline 
\textbf{2} &  4  & 0  & 0  & 4  & 0  & 0   \\  \hline 
\textbf{3} &  2  & 2  & 0  & 2  & 2  & 0   \\ \hline 
\textbf{4} & 0  & 4  & 0  & 0  & 4  & 0   \\ \hline 
\textbf{5}  & -2  & 6  & 0  & -2  & 6  & 0   \\ \hline 
\textbf{6}  & 0  &  2  & 2  & 0  & 2  & 2   \\ \hline 
\textbf{7}  & 0  & -3  & 5  & 0  & 5  & 1   \\ \hline 
\textbf{8}  & 0  & 0  & 4  & 0  & 0  & 4   \\ \hline 
\textbf{9}  & 3  & 0  & 1  & -1  & 0  & 5   \\ \hline 
\textbf{10}  & 0  & -2  & 6  & 0  & -2  & 6   \\ \hline 
\textbf{11}  & 1  & -3  & 4  & -3  & 5  & 4   \\ \hline 
\textbf{12}  & 2  & 0  & 0  & 2  & -4  & 8   \\ \hline 
\textbf{13}  & 0  & 2  & 0  & 0  & -2  & 8   \\ \hline 
\textbf{14}  & -3  & 9  & 0  & 1  & 1  & 0   \\ \hline
\textbf{15}  & 0  & 0  & 2  & 0  & -4  & 10   \\ \hline
\caption{Power of $\eta$-functions being basis elements for $S_{4}(\Gamma_{0}(44))$}
\label{convolutionSums-4_11-table}
\end{longtable}

\end{document}